\documentclass[12pt,twoside,reqno]{amsart}
\usepackage[colorlinks=true,citecolor=blue]{hyperref}
\usepackage{mathptmx, amsmath, amssymb, amsfonts, amsthm, mathptmx, enumerate, color}
\usepackage{graphicx}
\usepackage{epstopdf}

\newtheorem{theorem}{Theorem}[section]
\newtheorem{corollary}{Corollary}[section]

\newtheorem{proposition}{Proposition}[section]
\theoremstyle{definition}
\newtheorem{definition}{Definition}[section]
\newtheorem{example}{Example}[section]
\newtheorem{remark}{Remark}[section]

\numberwithin{equation}{section}

\begin{document}
\setcounter{page}{1}

\vspace*{1.0cm}
\title[Relations between NVVI and NVOP on Hadamard manifold in terms of bifunction]
{Relations between nonsmooth vector variational inequalities and nonsmooth vector optimization problems on Hadamard manifold in terms of bifunction}
\author[N. Singh, A. Iqbal and S. Ali]{ Nagendra Singh$^{1}$, Akhlad Iqbal$^{1,*}$, Shahid Ali$^{1}$}
\maketitle
\vspace*{-0.6cm}

\begin{center}
{\footnotesize {\it

$^1$Department of Mathematics, Aligarh Muslim University, Aligarh, India\\

}}\end{center}

\vskip 4mm {\small\noindent {\bf Abstract.}
In this paper, we discuss the concepts of bifunction and geodesic convexity for vector valued functions on Hadamard manifold. The Hadamard manifold is a particular type of Riemannian manifold with non-positive sectional curvature. Using bifunction, we introduce a definitions of generalized  geodesic convexity in the context of the Hadamard manifold. To support the definition, we construct a non-trivial example that demonstrates the property of geodesic convexity on Hadamard manifold. Additionally, we define the geodesic $h$-convexity, geodesic $h$-pseudoconvexity and geodesic $h$-quasiconvexity for vector valued function using bifunction and study their several properties. Furthermore, we demonstrate the uniqueness of the solution for nonsmooth vector variational inequality problem (NVVIP) and proved the characterization property for the solution of NVVIP and the Minty type NVVIP (MNVVIP) on Hadamard manifold in terms of bifunction. Afterwards, we consider a nonsmooth vector optimization problem (NVOP) and investigate the relationships among the solutions of NVOP, NVVIP and MNVVIP.

 }

\renewcommand{\thefootnote}{}
\footnotetext{ $^*$Akhlad Iqbal.
\par
E-mail addresses: singh.nagendra096@gmail.com, akhlad6star@gmail.com, shahid\_rrs@yahoo.com.
\par
Received January 23, 2015; Accepted February 16, 2016. }

\section{Introduction}


	In this paper, we discuss the concepts of bifunction and geodesic convexity for vector valued functions on Hadamard manifold. The Hadamard manifold is a particular type of Riemannian manifold with non-positive sectional curvature. Using bifunction, we introduce a definitions of generalized  geodesic convexity in the context of the Hadamard manifold. To support the definition, we construct a non-trivial example that demonstrates the property of geodesic convexity on Hadamard manifold. Additionally, we define the geodesic $h$-convexity, geodesic $h$-pseudoconvexity and geodesic $h$-quasiconvexity for vector valued function using bifunction and study their several properties. Furthermore, we demonstrate the uniqueness of the solution for nonsmooth vector variational inequality problem (NVVIP) and proved the characterization property for the solution of NVVIP and the Minty type NVVIP (MNVVIP) on Hadamard manifold in terms of bifunction. Afterwards, we consider a nonsmooth vector optimization problem (NVOP) and investigate the relationships among the solutions of NVOP, NVVIP and MNVVIP.

\noindent
Suppose the objective function is not smooth but has some directional derivatives. In that case, the VIP defined in terms of a bifunction can be used as a tool to investigate the NVOP. In 1980, Giannessi \cite{Giannessi} presented vector valued variational inequalities in finite dimensional Euclidean spaces. Lee et al. \cite{Lee} have established the relation between different kinds of VVIP and NVOP. Furthermore, Chen et al. \cite{Chen} extended the concept of VVIP and NVOP to the Hadamard manifolds using subdifferentials. For the latest work on vector optimization problems, see \cite{Chen,Chen1,Yang}.
\noindent
Rapscak \cite{Rapscak} and Udriste \cite{Udriste} took into consideration geodesic convexity, a generalization of convexity from the Euclidean space to Riemannian space. The concepts of convexity and monotonicity for set-valued function on Hadamard manifolds were studied by Barani et al. \cite{Barani}. Later, Zhou et al. \cite{Zhou} investigated the relation between VOP and VVIP on Hadamard manifolds.\\ 

\noindent
Motivated by the work of \cite{Ansari3,Chen,Crespi,Nemeth}, in the present article, we define  NVVI and MNVVI in terms of a vector version bifunction. We also define the geodesic convex (GC), geodesic h-convexity and generalized class of h-convexity of vector valued functions on Hadamard manifolds and discuss some relations between the class of geodesic convexity and geodesic h-convexity of vector valued functions. Afterwords, we define the class of monotonicity of vector valued bifunction on Hadamard manifolds that helps to obtain the relation between NVVIP and MNVVIP. Furthermore, we consider a NVOP and discuss the relationships among the solutions of NVOP, NVVIP and MNVVIP. 

\section{Preliminaries}

Let $\mathbb{R}^n$ represents an Euclidean space of dimension $ n $ with the usual norm and $\mathbb{R}^n_{+}$ be the non-negative orthant of $\mathbb{R}^n$. We use the following standards notations for vectors in $ \mathbb{R}^n$ throughout the paper.\\
Let $ p=(p_{1},p_{2},...,p_{n}) $ and $ q=(q_{1},q_{2},...,q_{n}) $ be two vectors in $ \mathbb{R}^{n} $. Then,
\begin{equation*}
	p\leqq q\Leftrightarrow p_{k}\leq q_{k}\quad \mbox{for}~k=1,2,...,n\quad\Leftrightarrow p-q\in -\mathbb{R}_{+}^{n};
\end{equation*}
\noindent
\begin{equation*}
	p\leq q \Leftrightarrow p_{k}\leq q_{k}\quad \mbox{for}\quad k=1,2,...,n \quad \mbox{and} \quad p\neq q \quad\Leftrightarrow p-q \in -\mathbb{R}_{+}^n\setminus\{0\};
\end{equation*}
\noindent
\begin{equation*}
	p<q \Leftrightarrow p_{k}< q_{k}\quad \mbox{for}\quad k=1,2,...,n \quad\Leftrightarrow p-q \in -int\mathbb{R}_{+}^n; 
\end{equation*}
\noindent
Let $ \mathbb{M} $ represents a Riemannian manifold of dimension $ n $ with Levi-civita  (or Riemannian) connection $ \nabla $ in throughout the paper. Let $\langle \cdot,\cdot\rangle_{p} $ represents the scalar product on $ T_{p}\mathbb{M} $ with the associated norm $\| \cdot\|_{p}$,  subscript $ p $ may be omitted.\\
\noindent
For $ p,q\in \mathbb{M} $, let $ \gamma_{pq} :[0,1]\longrightarrow{\mathbb{M}} $ be a smooth curve from $ p $ to $ q $ and $ \dot\gamma_{pq}(t) $ be the tangent vector to the curve $ \gamma_{pq} $. Then the arc length of curve $ \gamma_{pq}(t) $ is $ d(p,q)=inf_{\gamma_{pq}}L(\gamma_{pq}) $.\\
Where,
\begin{equation*}
	L(\gamma_{pq}):= \int_{0}^{1}\|\dot\gamma_{pq}(t)\|dt,
\end{equation*}\\
\noindent
The smooth curve $\gamma_{pq}$ is called geodesic if
\begin{equation*}
	\gamma_{pq}(0)=p\quad\gamma_{pq}(1)=q\quad\mbox{and}\quad \nabla_{\dot\gamma_{pq}}{\dot\gamma_{pq}}=0\quad\mbox{on}\quad[0,1].
\end{equation*}
Hopf-Rinow theorem states that if any two points on $ \mathbb{M}$ can be joined by a minimal geodesic, then the Riemannian manifold $ \mathbb{M} $ is complete.\\
For complete manifold, the exponential map $ \exp_{p}:T_{p}\mathbb{M}\longrightarrow \mathbb{M} $ at point $ p $ is defined by $ \exp_{p}v=\gamma_{v}(1,p) $, $\forall~v\in \mathbb{M} $.
Where $ \gamma(\cdot)=\gamma_{v}(\cdot,p) $ is the geodesic emanating from $ p $ with the velocity $ v $, $ i.e$ $ \gamma(0)=p $ and $ \dot\gamma(0)=v $. In general $ \exp_{p}(tv)=\gamma(t,v) $ for every real number $ t $.
The parallel transport along $\gamma_{pq}$ on the tangent bundle $ T\mathbb{M} $ with respect to $\nabla$ is denoted by $ P_{\gamma_{pq},.,.} $ and defined as 
\begin{equation*}
	P_{\gamma_{pq},\gamma_{pq}(0),\gamma_{pq}(1)}(v)=V(\gamma_{pq}(1))\quad \mbox{for all}\quad v\in T_{\gamma_{pq}(0)}M,
\end{equation*}
where $ V $ is the unique vector field satisfying $ \nabla_{\dot\gamma_{pq}(t)}V=0$, $\forall$ $ t $ and $ V(\gamma_{pq}(0))=v $.\\
In generally, The parallel transport on $ \mathbb{H} $ from $ T_{p}\mathbb{H} $ to $ T_{w}\mathbb{H} $ is represented by $ P_{w,p} $ .
\begin{definition}{\rm\cite{Udriste}}\label{d2.1}
	Let $ S $ $\subseteq$ $ \mathbb{M} $ and if for every points $ p,q \in S$, $\exists$ a geodesic $ \gamma_{pq}(t) $ joining $ p $ to $ q $ lying in $ S $, then $ S $ is called geodesic convex $(GC)$.
\end{definition}
\noindent
$ \Gamma_{pq}^{S} $ represents the set of all geodesics joining any two distinct points $ p,q \in  S $ and lying in $ S $.\\
\begin{definition}{\rm\cite{Udriste}}\label{d2.2}
	Let $ S(\neq\phi) $ $\subseteq$ $ \mathbb{M} $ be set and $ \Phi:S\longrightarrow \mathbb{R} $ be a function. If for any $ p,q \in S$ and any geodesic $ \gamma_{pq}\in S $, we have 
	\begin{equation*}
		\Phi(\gamma_{pq}(t))\leq (1-t)\Phi(p)+t\Phi(q),\quad\forall~ t\in [0,1],
	\end{equation*}
	then $ \Phi $ is called $ GC $.
\end{definition} 
\noindent
A manifold is called Hadamard manifold if it is complete simply connected and the sectional curvature is nonpositive. In this paper, we denote by $ \mathbb{H} $ the Hadamard manifold. 
\begin{proposition}{\rm\cite{Docarmo}}\label{p2.1}
	For any $ p\in \mathbb{H} $. The  exponential map $ \exp_{p}:T_{p}\mathbb{H}\longrightarrow \mathbb{H} $ is a diffeomorphism, and $\forall$ $ p,q\in \mathbb{H} $, there is a unique minimal geodesic $ \gamma_{pq}$ joining $ p $ to $ q $ such that 
	\begin{equation*}
		\gamma_{pq}(t)=\exp_{p}(t\exp_{p}^{-1}q),\quad\forall~t\in [0,1].
	\end{equation*} 
\end{proposition}
\noindent
As \cite{Ansari3}, for a fixed $ s\in (0,1) $, let $ w= \gamma_{pq}(s)=\exp_{p}(s\exp_{p}^{-1}q) $ be a point on the geodesic $ \gamma_{pq}:[0,1]\longrightarrow \mathbb{H} $ such that the geodesic is divided by $ w $ into two parts. First part can be written as
\begin{equation*}
	\gamma_{wp}(t)=\gamma_{pq}(-st+s)=\exp_{p}(-st+s)\exp_{p}^{-1}q,\quad \forall~t\in [0,1].
\end{equation*}
that is,
\begin{equation}\label{e1}
	\exp_{w}(t\exp_{w}^{-1}p)=\exp_{p}(-st+s)\exp_{p}^{-1}q,\quad\forall~t\in [0,1],
\end{equation}
and the second part can be written as 
\begin{equation*}
	\gamma_{wq}=\gamma_{pq}((1-s)t+s)=\exp_{p}(((1-s)t+s)\exp_{p}^{-1}q)\quad\forall~t\in [0,1],
\end{equation*}
that is,
\begin{equation}\label{e2}
	\exp_{w}(t\exp_{w}^{-1}q)=\exp_{p}(((1-s)t+s)\exp_{p}^{-1}q)\quad\forall~t\in [0,1].
\end{equation}
From (\ref{e1}), we get
\begin{equation}\label{e3}
	\exp_{w}^{-1}p=-sP_{w,p}\exp_{p}^{-1}q,
\end{equation}
Using (\ref{e2}), we get
\begin{equation}\label{e4}
	\exp_{w}^{-1}q=(1-s)P_{w,p}\exp_{p}^{-1}q.
\end{equation}
Similarly, we have 
\begin{equation}\label{e5}
	\exp_{w}^{-1}p=sP_{w,q}\exp_{q}^{-1}p.
\end{equation}
\begin{definition}{\rm\cite{Ferreira}}\label{d2.4}
	A finite function $ \Phi:\mathbb{H} \longrightarrow \mathbb{R}\cup\{\pm\infty\} $ is defined on $ \mathbb{H} $. Then,
	\begin{enumerate}
		\item If for a point $ p $
		\begin{equation*}
			\Phi^{D}(p;v)=\limsup_{t\rightarrow0^{+}}\frac{\Phi(\exp_{p}tv)-\Phi(p)}{t}.
		\end{equation*}
		Then $ \Phi $  is called the Dini upper directional derivative at $ p \in \mathbb{H} $ in the direction $ v\in T_{p}\mathbb{H} $.
		\item If for a point $ p $
		\begin{equation*}
			\Phi_{D}(p;v)=\liminf_{t\rightarrow 0^{+}}\frac{\Phi(\exp_{p}tv)-\Phi(p)}{t}.
		\end{equation*}
		Then, $ \Phi $ is called the Dini lower directional derivative at $ p\in \mathbb{H} $ in the direction $ v\in T_{p}\mathbb{H} $.
	\end{enumerate}
\end{definition}
\section{Generalized convexities for vector valued function and nonsmooth vector variational inequalities}

In the recent years, the concept of generalized convexities, as pseudo-convexity, quasiconvexity etc, was defined by using different types of generalized derivatives. Komlosi \cite{Komlosi} defined quasiconvexity by using Dini directional derivative. Ansari et al. \cite{Ansari3} introduced the theory of pseudoconvexity and quasiconvexity of mapping by using a bifunction on Riemannian manifold. In this section, we extend the concept of convexity, pseudoconvexity, and quasiconvexity for vector valued function by using vector valued bifunction on Hadamard manifolds.
Firstly, we define geodesic convexity of vector valued functions on $ \mathbb{H} $ as follows:
\begin{definition}\label{d3.1}
	Let $ S(\neq\phi) $ $\subseteq$ $\mathbb{M}$ be $ GC $ set. A function $ \Phi:S\longrightarrow \mathbb{R}^{m} $ is said to be $ GC $ if for each $ p,q\in S $ and any geodesic $ \gamma_{pq}(t) $ joining $ p $ to $ q $, the function $ \Phi o\gamma_{pq}$ is convex on $ [0,1] $.\\
	That is,  $ a,b \in [0,1] $ and $ t\in [0,1] $
	\begin{equation*}
		\Phi o \gamma_{pq}(ta+(1-t)b)\leq t\Phi o\gamma_{pq}(a)+(1-t)\Phi o \gamma_{pq}(b);
	\end{equation*}
	or
	\begin{equation}\label{e6}
		\Phi o\gamma_{pq}(ta+(1-t)b)-t\Phi o\gamma_{pq}(a)-(1-t)\Phi o\gamma_{pq}(b)\in -\mathbb{R}_{+}^{m}\cup{\{0\}}.
	\end{equation}
\end{definition}
\begin{definition}\label{d3.2}
	Let $ S(\neq\phi) $ $\subseteq$ $ \mathbb{H} $ be a $ GC$ set, then a vector valued function $\Phi :S\longrightarrow \mathbb{R}^{m}$ is said to be geodesic quasiconvex if $\forall$ $ p,q \in S $, we have
	\begin{equation*}
		\Phi_{i}(w)\leq \max\{\Phi_{i}(p), \Phi_{i}(q)\},\quad\forall~i=1~to~p.
	\end{equation*}
	where, $\Phi=(\Phi_{1},\Phi_{2},...,\Phi_{m})$ and $ w=\exp_{q}(t\exp_{q}^{-1}p) $ for all $ t\in [0,1] $.
\end{definition}
\noindent
In the following example, we show the existence of vector valued geodesic convex function:
\begin{example}\label{ex3.3}
	Let  $\mathbb{H}=\{(x,y): y=\sqrt{1+x^{2}},~ x\in \mathbb{R} \} $ and $ p=(0,1) $ and $ q=(1,\sqrt{2}) $ be two points on $ \mathbb{H} $. The geodesic joining these points is $ \gamma_{pq}(t)=(t,\sqrt{1+t^2}) $. Then, a vector valued function $ \Phi:\mathbb{H}\longrightarrow\mathbb{R}^2 $ defined as $ \Phi(x,y)=(x^2,y^2) $ is convex on $ \mathbb{H} $. For this, we show the composition function $ \Phi o \gamma_{pq} $ is convex on $ [0,1] $ as follows:\\
	Let $ a, b \in[0,1] $ and $ t\in [0,1] $. Then,
	\begin{equation*}
		\Phi o\gamma_{pq}(ta+(1-t)b)-t\Phi o \gamma_{pq}(a)-(1-t)\Phi o \gamma_{pq}(b)=\Phi (\gamma_{pq}(ta+(1-t)b))-t\Phi(\gamma_{pq}(a))-(1-t)\Phi(\gamma_{pq}(b)),
	\end{equation*}
	\begin{equation*}
		~~~~~~~~~~~~~~~~~~~~~~~~~~~~~~~=\Phi(ta+(1-t)b,\sqrt{1+\{ta+(1-t)b\}^2})-t\Phi(a,\sqrt{1+a^2})-(1-t)\Phi(b,\sqrt{1+b^2}),
	\end{equation*}
	\begin{equation*}
		=-t(1-t)((a-b)^2,(a-b)^2)\leq0,
	\end{equation*}
	or
	\begin{equation*}
		\Phi o\gamma_{pq}(ta+(1-t)b)-t\Phi o\gamma_{pq}(a)-(1-t)\Phi o\gamma_{pq}(b)\in -\mathbb{R}_{+}^{2}\cup\{0\}.
	\end{equation*}
	That is, $ \Phi  $ is geodesic convex vector valued function on $ \mathbb{H} $.
\end{example}
\noindent
Let $ \mathbb{R}^{*}=\mathbb{R}\cup\{+\infty,-\infty\} $ denotes the set of extended real numbers. The concepts of NVVIP and MNVVIP on Hadamard manifolds using bifunction were introduced by Singh at al. \cite{Singh} as follows. They \cite{Singh} also discussed their existence. 
\begin{definition}\label{d3.4}\cite{Singh}
	Suppose $ S\subseteq\mathbb{H} $ and let $ h:S \times T\mathbb{H}\longrightarrow (\mathbb{R}^{*})^{m} $ be  vector valued bifunction defined as
	\begin{equation*}
		h(p, \exp_{p}^{-1}q)=(h_{1}(p, \exp_{p}^{-1}q),h_{2}(p, \exp_{p}^{-1}q),...,h_{m}(p, \exp_{p}^{-1}q));
	\end{equation*}
	where, $ h_{i}:S\times T\mathbb{H}\longrightarrow\mathbb{R}\cup\{+\infty,-\infty\} $ be a bifunction see \cite{Ansari3}. Then,
	\begin{enumerate}
		\item (NVVI): Find $\bar{p}\in S$, such that
		\begin{equation*}
			h(\bar{p},\exp_{\bar{p}}^{-1}q)=(h_{1}(\bar{p},\exp_{\bar{p}}^{-1}q),h_{2}(\bar{p},\exp_{\bar{p}}^{-1}q),...,h_{m}(\bar{p},\exp_{\bar{p}}^{-1}q))\notin -\mathbb{R}_{+}^{m} \setminus \{0\};
		\end{equation*}
		\item (MNVVI): Find $\bar{p}\in S $, such that 
		\begin{equation*}
			h(q,\exp_{q}^{-1}\bar{p})=(h_{1}(q,\exp_{q}^{-1}\bar{p}),h_{2}(q,\exp_{q}^{-1}\bar{p}),...,h_{m}(q,\exp_{q}^{-1}\bar{p}))\notin\mathbb{R}_{+}^{m}\setminus\{0\}.
		\end{equation*}
	\end{enumerate}
\end{definition}
\noindent
convexity and monotonicity were defined by  Ansari \cite{Ansari3} in terms of real valued bifunction on Riemannian manifold. We use vector valued bifunction to generalize those concepts for vector valued function on Hadamard manifold in the manner described below:
\begin{definition}\label{d3.5}
	Let $ S(\neq\phi)\subseteq \mathbb{H} $ be a $ GC $ set and $ h:S\times T\mathbb{H}\longrightarrow (\mathbb{R}^{*})^{m} $ be a vector valued bifunction. Then, a function $ \Phi:S\longrightarrow\mathbb{R}^{m} $ is said to be
	\begin{enumerate}
		\item Geodesic h-convex if $\forall$ $ p,q\in S $ with $ p\neq q $, 
		\begin{equation*}
			h(p;\exp_{p}^{-1}q)\leq\Phi(q)-\Phi(p);
		\end{equation*}
		or
		\begin{equation*}
			h(p;\exp_{p}^{-1}q)+\Phi(p)-\Phi(q)\notin\mathbb{R}_{+}^{m}\setminus \{0\}.
		\end{equation*}
		\item Geodesic $h$-pseudoconvex if $\forall$ $ p,q \in S $ with $ p\neq q $,
		\begin{equation*}
			\Phi(q)\textless\Phi(p)\Rightarrow h(p;\exp_{p}^{-1}q)\in -int\mathbb{R}_{+}^{m}.
		\end{equation*}
		\item Geodesic h-quasiconvex if $\forall$ $ p,q \in S $ with $ p\neq q $, 
		\begin{equation*}
			\Phi(q)\leq \Phi(p)\Rightarrow h(p;\exp_{p}^{-1}q)\notin\mathbb{R}_{+}^{m}\setminus\{0\}.
		\end{equation*}
	\end{enumerate}
\end{definition}
\begin{theorem}\label{th3.6}
	Let $ S(\neq\phi) \subseteq \mathbb{H} $ be a $ GC $ set and $ h:S\times T\mathbb{H}\longrightarrow (\mathbb{R}^{*})^{m} $ be a vector valued bifunction and let $ \Phi: S\longrightarrow \mathbb{R}^{m} $ be a geodesic quasiconvex function with the condition that for each $ i= 1~to~m $,
	\begin{equation}\label{e7}
		h_{i}(p;\exp_{p}^{-1}q)\leq \Phi_{i}^{D}(p,\exp_{p}^{-1}q),
	\end{equation} 
	then $ \Phi $ is geodesic h-quasiconvex.
\end{theorem}
\begin{proof}
	Since $ \Phi $ is geodesic quasiconvex, $\forall$ $ p,q \in S $, and $ p\neq q $ we have
	\begin{equation*}
		\Phi(q)\leq \Phi(p)\Rightarrow \Phi (\exp_{p}t\exp_{p}^{-1}q)\leq\Phi(p),\quad t \in [0,1].
	\end{equation*}
	\begin{equation*}
		\Rightarrow  \Phi_{i}^{D}(p,\exp_{p}^{-1}q)\leq0,\quad\forall~i=1~to~m. 
	\end{equation*}
	By the equation (\ref{e7}), we have
	\begin{equation*}
		h_{i}(p,\exp_{p}^{-1}q)\leq 0\quad\forall~i.
	\end{equation*}
	Hence,
	\begin{equation*}
		h(p,\exp_{p}^{-1}q)\notin\mathbb{R}_{+}^{m}\setminus\{0\}
	\end{equation*}
	$\Rightarrow$ $ \Phi $ is h-quasiconvex function.
\end{proof}
\noindent
The following results in the immediate consequence of Theorem \ref{th3.6}.
\begin{corollary}\label{cor3.7}
	Let $ S(\neq\phi) \subseteq\mathbb{H} $ be a $ GC $ set and $ h:S\times T\mathbb{H}\longrightarrow (\mathbb{R}^{*})^{m} $ be a vector valued bifunction and $ \Phi:S\longrightarrow\mathbb{R}^{m} $ be a function satisfying equation (\ref{e7}) and that
	\begin{equation}\label{e8}
		h(p,v)\geqq-h(p,-v)\quad ,\quad h(p,\alpha v)=\alpha h(p,v)\quad\forall~v\in T_{p}\mathbb{H},~p\in S,~\alpha\textgreater0.
	\end{equation}
	If $ \Phi $ is h-pseudoconvex function then it is geodesic quasiconvex and hence geodesic h-quasiconvex.
\end{corollary}
\begin{proof}
	On contrary assume that $\exists~p,q,w\in S$ with $ p\neq q $ and that $ w=\exp_{p}t\exp_{x}^{-1}q\quad\forall~t\in [0,1] $. there exists an $ i $ such that
	\begin{equation*}
		\Phi_{i}(w)=\max\{\Phi_{i}(p),\Phi_{i}(q)\}\quad i=1~to~m.
	\end{equation*}
	Since $\Phi$ is geodesic h-pseudoconvex function we have
	\begin{equation*}
		h_{i}(w,\exp_{w}^{-1}p)\textless0\quad\mbox{and}\quad h_{i}(w,\exp_{w}^{-1}q)\textless0\quad\forall~i.
	\end{equation*}
	Using equations (\ref{e3}) and (\ref{e4}) we get
	\begin{equation*}
		h_{i}(w,-P_{w,p}\exp_{p}^{-1}q)+h_{i}(w,P_{w,p}\exp_{p}^{-1}q)\in -int\mathbb{R}_{+}^{m}
	\end{equation*}
	which is contradiction to equation (\ref{e8}). Therefore $\Phi$ is geodesic quasiconvex and by using theorem \ref{th3.6}, $\Phi$ is geodesic h-convex function. 
\end{proof}

\noindent
Monotonicity of vector valued bifunctions on Euclidean space is very useful tool to prove uniqueness of solution of NVVIP and it establish some relations between NVVIP and MNVVIP see \cite{Ansari1} . Now, we extend the concept of monotonicity for the vector valued bifunction on $\mathbb{H}$.
\begin{definition}\label{d4.1}
	A vector valued bifunction $ h:S\times T\mathbb{H}\longrightarrow (\mathbb{R}^{*})^{m} $ is said to be 
	\begin{enumerate}
		\item monotone if $\forall$ $ p,q\in S $ with $ p\neq q $, we have
		\begin{equation*}
			h(p, \exp_{p}^{-1}q)+h(q,\exp_{q}^{-1}p)\notin \mathbb{R}_{+}^{m}\setminus \{0\}.
		\end{equation*}
		\item pseudomonotone if $\forall$ $ p,q\in S $ with $ p\neq q $, we have
		\begin{equation*}
			h(p,\exp_{p}^{-1}q)\notin -\mathbb{R}_{+}^{m}\setminus\{0\}\Rightarrow h(q,\exp_{q}^{-1}p)\notin \mathbb{R}_{+}^{m}\setminus\{0\}.
		\end{equation*}
		\item strictly pseudomonotone if $\forall$ $ p,q\in S $ with $ p\neq q $, we have
		\begin{equation*}
			h(p,\exp_{p}^{-1}q)\notin -\mathbb{R}_{+}^{m}\setminus \{0\}\Rightarrow h(q,\exp_{q}^{-1}q)\in -int\mathbb{R}_{+}^{M}.
		\end{equation*}
	\end{enumerate}
\end{definition}
\begin{example}
	Let $\mathbb{R}_{++}=\{p\in\mathbb{R}: p>0\}$ and $\mathbb{H}=(\mathbb{R}^n_{++},g)$ be a Hadamard manifold with null sectional curvature, where $g(u,v)=\frac{1}{p^2}uv,~\forall~p\in \mathbb{H}\quad\mbox{and}\quad u,v\in T_{p}\mathbb{H}$ is a Riemannian metric.\\
	In particular, $n=1$, $\mathbb{H}=(\mathbb{R}_{++}, g)$, For $p\in \mathbb{H}$, tangent plane $T_{p}\mathbb{H}$ at $p$ equal to $\mathbb{R}$. The unique geodesic $\gamma$ starting from $p=\gamma(0)\in \mathbb{H}$ with the velocity $\gamma^\prime(0)=v\in T_{p}\mathbb{H}$ is defined by
	\begin{eqnarray*}
		\gamma(t)=pe^{(\frac{v}{p})t}
	\end{eqnarray*}
	then 
	\begin{eqnarray*}
		\exp_{p}tv=pe^{(\frac{v}{p})t}\quad\mbox{and}\quad\exp_{p}^{-1}q=p\ln\frac{q}{p}.
	\end{eqnarray*}
	Now, Let $S=\mathbb{R}_{++}$ and we define $h:S\times T\mathbb{H}\longrightarrow (\mathbb{R}^{*})^2$ as
\end{example}
\begin{eqnarray*}
	h(p;d)=\left(-p,\frac{(\ln p-1)}{p}\cdot d\right)
\end{eqnarray*}
Now, 
\begin{eqnarray*}
	h(p;\exp_{p}^{-1}q)+h(q;\exp_{q}^{-1}p)=\left(-p,\frac{(\ln p-1)}{p}\cdot p\ln\frac{q}{p}\right)+\left(-q, \frac{(\ln q-1)}{q}\cdot q\ln \frac{p}{q}\right)
\end{eqnarray*}
\begin{eqnarray*}
	~~~~~~~~~~~~~~~~~~~~~~~~~~= \left(-p-q, \ln p\ln q-\ln q-\ln p\ln p+\ln p +\in q\ln p-\ln p-\ln q\ln q+\ln q\right)
\end{eqnarray*}
\begin{eqnarray*}
	~~~~~~~~~~~~~~~~~~~~~~~~~~=\left(-p-q,-(\ln p-\ln q)^{2}\right).
\end{eqnarray*}
\begin{eqnarray*}
	h(p;\exp_{p}^{-1}q)+h(q;\exp_{q}^{-1}p)\notin\mathbb{R}^{2}_{+}\setminus\{0\}
\end{eqnarray*}
Hence, $h$ is monotone.\\
\noindent
By considering strict pseudomonotonicity of the bifunction $ h $, we investigate the uniqueness of the of NVVIP in the ensuing theorem:
\begin{theorem}
	Let $ h:S\times T\mathbb{H}\longrightarrow(\mathbb{R}^{*})^{m} $ be strictly pseudomonotone bifunction then NVVIP has a unique solution if it exists. 
\end{theorem}
\begin{proof}
	On contrary assume that $ \bar{p} $, $ \hat{p} $ $ \in S $ be two distinct solutions of NVVIP. Then
	\begin{equation*}
		h(\bar{p},\exp_{\bar{p}}^{-1}q)\notin -\mathbb{R}_{+}^{m}\setminus\{0\}\quad\mbox{and}\quad h(\hat{p},\exp_{\hat{p}}^{-1}q)\notin-\mathbb{R}_{+}^{m}\setminus\{o\}\quad\forall~q\in \mathbb{H}.
	\end{equation*}
	Now, put $ q=\hat{p} $ and $ \bar{p} $ respectively, we get
	\begin{equation}\label{e17}
		h(\bar{p},\exp_{\bar{p}}^{-1}\hat{p})\notin -\mathbb{R}_{+}^{m}\setminus\{0\}.
	\end{equation}
	And also,
	\begin{equation}\label{e18}
		h(\hat{p},\exp_{\hat{p}}^{-1}\bar{p})\notin-\mathbb{R}_{+}^{m}\setminus \{0\}.
	\end{equation}
	Using strict pseudomonotonicity of $ h $, we have
	\begin{equation*}
		h(\hat{p},\exp_{\hat{p}}^{-1}\bar{p})\in -int\mathbb{R}_{+}^{m}.
	\end{equation*}
	Which is a contradiction to equation (\ref{e18}). And hence $\bar{p}=\hat{p} $.
\end{proof}

\noindent
In the following definition, we extend the concept of upper sign continuity  \cite{Ansari1} on $ \mathbb{R}^{n} $ to Hadamard manifold:
\begin{definition}\label{d4.3}{\rm\cite{Singh}}
	\noindent 
	Let $ S\subseteq\mathbb{H} $ and $ h:S\times T\mathbb{H}\longrightarrow (\mathbb{R}^{*})^{m} $ be a vector valued bifunction is said to be geodesic upper sign continuous If for any pair of distinct points $ p,q\in S $ and $ t\in [0,1] $
	we have
	\begin{equation*}
		h(w_{t}, P_{w_{t},q}\exp_{q}^{-1}p)\notin \mathbb{R}_{+}^{m}\setminus\{0\}\Rightarrow h(p,\exp_{p}^{-1}q)\notin -\mathbb{R}_{+}^{m}\setminus\{0\}.
	\end{equation*}\\
	Where $ w_{t}=\exp_{p}t\exp_{p}^{-1}q $.	
\end{definition}
\noindent
In the following theorem, we discuss an important characterization for the solution of NVVIP and MNVVIP. 
\begin{theorem}\label{th4.4}
	Let $ S \subseteq \mathbb{H} $ be a nonempty $ GC $ set and $ h:S\times T\mathbb{H}\longrightarrow (\mathbb{R}^{*})^{m} $ be pseudomonotone and geodesic upper sign continuous vector valued bifunction which is positively homogeneous in the second argument . Then, $ \bar{p}\in S $ is the solution of NVVIP iff it is a solution of MNVVIP.
\end{theorem}
\begin{proof}
	By the pseudomonotonicity of $ h $,
	\begin{equation*}
		h(\bar{p},\exp_{\bar{p}}^{-1}q)\notin -\mathbb{R}_{+}^{m}\setminus\{0\}\Rightarrow h(q,\exp_{q}^{-1}\bar{p})\notin \mathbb{R}_{+}^{m}\setminus \{0\}.
	\end{equation*}
	That is,
	\begin{equation*}
		h(q,\exp_{q}^{-1}\bar{p})=(h_{1}(q,\exp_{q}^{-1}\bar{p}),...,h_{m}(q,\exp_{q}^{-1}\bar{p}))\notin \mathbb{R}_{+}^{m}\setminus \{0\}.
	\end{equation*}
	This implies that $ \bar{p} $ is a solution of MNVVIP.\\
	Conversely, suppose $ \bar{p}\in S $ is a solution of MNVVIP. Then,
	\begin{equation*}
		h(q,\exp_{q}^{-1}\bar{p})\notin \mathbb{R}_{+}^{m}\setminus\{0\}\quad\forall~q\in S.
	\end{equation*}
	Since $ S $, is GC, therefore $\forall$ $ t\in (0,1)$, we have
	$ q_{t}=\exp_{\bar{p}}t\exp_{\bar{p}}^{-1}q\in S $
	\begin{equation*}
		h(q_{t},\exp_{q_{t}}^{-1}\bar{p})\notin \mathbb{R}_{+}^{m}\setminus\{0\}.
	\end{equation*}
	By using equation (\ref{e5}), we obtain
	\begin{equation*}
		h(q_{t},tP_{q_{t},q}\exp_{q}^{-1}\bar{p})\notin \mathbb{R}_{+}^{m}\setminus \{0\}.
	\end{equation*}
	Further, by definition (\ref{d4.3})
	\begin{equation*}
		h(\bar{p},\exp_{\bar{p}}^{-1}q)\notin -\mathbb{R}_{+}^{m}\setminus\{0\}.
	\end{equation*}
	That is, $ \bar{p} $ is a solution of NVVIP.
\end{proof}
\begin{theorem}\label{p4.5}
	Let $ h:S\times T\mathbb{H}\longrightarrow (\mathbb{R}^{*})^{m} $ be a vector valued bifunction which is subadditive  and positively homogeneous in second argument. For $ p\in S\subseteq\mathbb{H}$ , if we set $ W=\{q\in S:h(p,\exp_{p}^{-1}q)\in -int\mathbb{R}_{+}^{m}\} $. Then $\forall$ $ p\in S$, the set $ \exp_{p}^{-1}W $ is geodesic convex.
\end{theorem}
\begin{proof}
	To prove $ \exp_{p}^{-1}W $ is $GC$ for every $ p\in S $, we have to show that, for $ 0\leq t \leq 1 $ and $ q_{1}, q_{2}\in W $ 
	\begin{equation*}
		\exp_{p}[(1-t)\exp_{p}^{-1}q_{1}+t\exp_{p}^{-1}q_{2}]\in W,
	\end{equation*}
	since, $ q_{1}, q_{2}\in W $, we have
	\begin{equation}\label{e19}
		h(p,\exp_{p}^{-1}q_{1})\in -int\mathbb{R}_{+}^{m},
	\end{equation}
	and
	\begin{equation}\label{e20}
		h(p,\exp_{p}^{-1}q_{2})\in -int\mathbb{R}_{+}^{m}.	
	\end{equation}
	On adding the resultant of (\ref{e19}) and (\ref{e20}) after multiplying by $ (1-t) $ and $ t $ respectively. And using the given assumption, we get
	\begin{equation*}
		h(p,\exp_{p}^{-1}\exp_{p}[(1-t)\exp_{p}^{-1}q_{1}+t\exp_{p}^{-1}q_{2}])\in -int \mathbb{R}_{+}^{m}.
	\end{equation*}
	This implies
	\begin{equation*}
		\exp_{p}[(1-t)\exp_{p}^{-1}q_{1}+t\exp_{p}^{-1}q_{2}]\in W.
	\end{equation*}
	Hence,
	$ \exp_{p}^{-1}W $ is geodesic convex set.
\end{proof}
\noindent
\section{Nonsmooth vector optimization problem}
Let $ S(\neq\phi) \subseteq \mathbb{H} $ be a $ GC $ set and $\Phi:S\longrightarrow \mathbb{R}^{m}$ be a vector valued objective function. Here, we consider a following nonsmooth vector optimization problem
\begin{equation}\label{e9}
	\mbox{(NVOP)}\quad\quad Min~\Phi(p)=(\Phi_{1}(p),\Phi_{2}(p),...,\Phi_{m}(p))
\end{equation}
\begin{center}
	s.t. $ p\in S. $\\
\end{center}
Where, $ \Phi_{i}:S\longrightarrow\mathbb{R} $ be a real valued function, for $i= 1~to~m $.
\noindent
\begin{definition}
	A point $ \bar{p}\in \mathbb{H} $ is said to be 
	\begin{enumerate}
		\item an efficient solution of NVOP if
		\begin{equation}\label{e10}
			\Phi(q)-\Phi(\bar{p})=(\Phi_{1}(q)-\Phi_{1}(\bar{p}),\Phi_{2}(q)-\Phi_{2}(\bar{p}),...,\Phi_{m}(q)-\Phi_{m}(\bar{p}))\notin-\mathbb{R}_{+}^{m}\setminus\{0\}\quad\forall~q\in S.
		\end{equation}
		\item weakly efficient solution of NVOP if 
		\begin{equation}\label{e11}
			\Phi(q)-\Phi(\bar{p})=(\Phi_{1}(q)-\Phi_{1}(\bar{p}),..., \Phi_{m}(q)-\Phi_{m}(\bar{p}))\in int\mathbb{R}_{+}^{m}\quad\forall~q\in S.
		\end{equation}
	\end{enumerate} 
\end{definition}
\begin{theorem}
	Let $ S(\neq\phi) \subseteq \mathbb{H} $ be a $ GC $ set and $ \Phi:S\longrightarrow\mathbb{R}^{m} $ be a vector valued function. Let a bifunction $ h:S\times T\mathbb{H}\longrightarrow(\mathbb{R}^{*})^{m} $ satisfies the condition 
	\begin{equation}\label{e12}
		\Phi_{i,D}(p,\exp_{p}^{-1}q)\leq h_{i}(p,\exp_{p}^{-1}q),\quad\forall~i.
	\end{equation}
	If $ \bar{p}\in S $ is an efficient solution of NVOP then it is a solution of NVVIP.
\end{theorem}
\begin{proof}
	Suppose $ \bar{p} $ is the solution of NVOP. That is,
	\begin{equation*}
		\Phi(q)-\Phi(\bar{p})=(\Phi_{1}(q)-\Phi_{1}(\bar{p}),...,\Phi_{m}(q)-\Phi_{m}(\bar{p}))\notin-\mathbb{R}_{+}^{m}\setminus\{0\},\quad\forall~q\in S.
	\end{equation*}
	Since, $ S $ is geodesic convex, so $ \exp_{\bar{p}}t\exp_{\bar{p}}^{-1}q\in S $,\quad$\forall$ $ t\in [0,1] $.\\
	Therefore, 
	\begin{equation*}
		\Phi(\exp_{\bar{p}}t\exp_{\bar{p}}^{-1}q)-\Phi(\bar{p})\notin-\mathbb{R}_{+}^{m}\setminus\{0\}.
	\end{equation*}
	\begin{equation*}
		\Rightarrow ~~\Phi_{i}(\exp_{\bar{p}}t\exp_{\bar{p}}^{-1}q)\geq\Phi_{i}(\bar{p})\quad\forall~i.
	\end{equation*}
	\begin{equation*}
		\Rightarrow~~\frac{\Phi_{i}(\exp_{\bar{p}}t\exp_{\bar{p}}^{-1}q)-\Phi_{i}(\bar{p})}{t}\geq 0\quad\forall~q\in S~\mbox{and}~t\in [0,1].
	\end{equation*}
	Now, taking $liminf$ as $ t\rightarrow0^{+} $ and using (\ref{e12}) we have
	\begin{equation*}
		h(\bar{p},\exp_{\bar{p}}^{-1}q)\notin-\mathbb{R}_{+}^{m}\setminus\{0\}.
	\end{equation*}
	Hence, $ \bar{p} $ is a solution of NVVIP.
\end{proof}

\noindent
\begin{theorem}
	Let $ \Phi:S\longrightarrow\mathbb{R}^{m} $ be a vector valued function s.t.
	\begin{equation}\label{e13}
		h(p,\exp_{p}^{-1}q)-\Phi(q)+\Phi(p)\in int\mathbb{R}_{+}^{m}\quad\forall p,q \in S,~p\neq q.
	\end{equation}
	Then, every weakly efficient solution of NVOP is the solution of NVVIP. 
\end{theorem}
\begin{proof}
	On contrary, suppose that $ \bar{p}\in S $ is a weakly efficient solution of the NVOP but is not the solution of NVVIP. Then, $\exists$ $ q\in S $ s.t.
	\begin{equation}\label{e14}
		h(\bar{p},\exp_{\bar{p}}^{-1}q)\in -int\mathbb{R}_{+}^{m}.
	\end{equation}
	From (\ref{e13}) and (\ref{e14})
	\begin{equation*}
		\Phi(q)-\Phi(\bar{p})\in -int\mathbb{R}_{+}^{m}
	\end{equation*}
	\begin{equation*}
		\Rightarrow \Phi_{i}(q)\textless\Phi_{i}(\bar{p})\quad\forall~i\in I.
	\end{equation*}
	Which is a contradiction to our assumption. Hence $ \bar{p} $ is the weakly efficient solution of NVOP.
\end{proof}
\begin{theorem}
	Let $ S(\neq\phi) \subseteq \mathbb{H} $ be a $ GC $ set and $\Phi:S\longrightarrow \mathbb{R}^{m} $ be a geodesic h-convex function. If $ \bar{p}\in S $ is the weakly efficient solution of NVOP then it is also solution of MNVVIP. 
\end{theorem}
\begin{proof}
	Let $ \bar{p}\in S $ be an efficient solution of NVOP, then
	\begin{equation}\label{e15}
		\Phi(q)-\Phi(\bar{p})=(\Phi_{1}(q)-\Phi_{1}(\bar{p}),...,\Phi_{m}(q)-\Phi_{m}(\bar{p}))\notin -\mathbb{R}_{+}^{m}\setminus\{0\}.
	\end{equation} 
	Since $ \Phi $ is geodesic h-convex, we have
	\begin{equation}\label{e16}
		h(q,\exp_{q}^{-1}\bar{p})+\Phi(q)-\Phi(\bar{p})\in -\mathbb{R}_{+}^{m}\setminus\{0\}.
	\end{equation}
	From equation (\ref{e15}) and (\ref{e16}), we get
	\begin{equation*}
		h(q,\exp_{q}^{-1}\bar{p})\notin\mathbb{R}_{+}^{m}\setminus\{0\}.
	\end{equation*}
	This implies that $ \bar{p} $ is the solution of MNVVIP.
\end{proof}
\begin{remark}
	We add the following condition  in order to relax the geodesic h-convexity of $\Phi$ by geodesic h-pseudoconvexity.
\end{remark}
\begin{theorem}\label{th3.13}
	Let $ S(\neq\phi) \subseteq \mathbb{H} $ be a $ GC $ set and a function $ \Phi:S\longrightarrow \mathbb{R}^{m} $ be a geodesic h-pseudoconvex. Let $ h:S\times T\mathbb{H}\longrightarrow (\mathbb{R}^{*})^{m} $ be a bifunction s.t.\\
	the equation (\ref{e7}) holds and that
	\begin{equation*}
		h(p,v)\geq -h(p,-v),~~h(p,\alpha v)=\alpha h(p,v)\quad\forall~~v\in T_{p}\mathbb{H}~\mbox{and}~p\in S.
	\end{equation*}
	Then, every efficient solution of NVOP is the solution of MNVVIP. 
\end{theorem}
\begin{proof}
	Let $ \bar{p}\in S $  is an efficient solution of NVOP,
	therefore
	\begin{equation*}
		\Phi(q)-\Phi(\bar{p})=(\Phi_{1}(q)-\Phi_{1}(\bar{p}),...,\Phi_{m}(q)-\Phi_{m}(\bar{p}))\notin -\mathbb{R}_{+}^{m}\setminus\{0\}\quad \forall~q\in S.
	\end{equation*}
	By corollary \ref{cor3.7}, $ \Phi $ is geodesic h-quasiconvex and we get
	\begin{equation*}
		h(q,\exp_{q}^{-1}\bar{p})\notin \mathbb{R}_{+}^{m}\setminus \{0\}
	\end{equation*}
	Therefore,
	$ \bar{p} $ is the solution of MNVVIP.
\end{proof}
\section{Conclusion}\label{sec13}
An efficient tool to study VOP is the theory of VVI, which was started by Giannessi \cite{Giannessi}. In this paper, we have extended the concept of NVVIP and MNVVIP on Hadamard manifolds by using bifunctions. The idea of geodesic convexity and geodesic quasiconvexity for the vector valued function have been investigated on the Hadamard manifolds. Several relations among different kinds of solutions of NVOP, NVVIP and MNVVIP have also been discussed. In future, the results of the paper can be explored to Riemannian manifolds. Equilibrium problems and Hierarchical problems can also be considered in the same manner as we have considered NVVIP, MNVVIP and NVOP for the vector valued functions.

\vskip 6mm
\noindent{\bf Acknowledgments}

\noindent   The author is grateful to the reviewers for useful suggestions which improved the contents of this paper.
The second author was supported by the  National Natural Science Foundation  under Grant No.
1244145e2.


\begin{thebibliography}{99}

	\bibitem{Akhtar}
Akhtar, A.K.; Tammer, C.; Zalinescu, C.: Sensitivity Analysis in Set-Valued Optimization and Vector Variational Inequalities. {\sl In book: Set-valued optimization,} (2015) 605-643.

\bibitem{Ansari1}
Ansari, Q.H.; Kobis, E.; Yao, J.C.: {\sl Vector variational Inequalities and vector optimization}. Springer nature, Switzerland, (2018).


\bibitem{Ansari3}
Ansari, Q.H.; Islam, M.; Yao, J.C.: Nonsmooth convexity and monotonicity in terms of bifunction on Riemannian manifolds. Journal of Nonlinear and Convex analysis, \textbf{18} (2017) 743-762.

\bibitem{Barani}
Barani, A.: Generalized monotonicity and convexity for locally Lipschitz functions on Hadamard manifolds. Diff. geometry-Dynamical syst. \textbf{15} (2013) 26-37.

\bibitem{Chen}
Chen, S.; Huang, N.: Vector variational inequalities and vector optimization problems on Hadamard manifolds. Optim. Lett. \textbf{10} (2016) 753-767.

\bibitem{Chen1}
Chen, S.; Fang, C.: Vector variational inequality with pseudoconvexity on Hadamard manifolds. Optimization \textbf{65} (2016) 2067-2080.

\bibitem{ChenG}
Chen, G.Y.; Craven, B.D.: A vector variational inequality and optimization over an efficient set, Z. Oper. Res. \textbf{34} (1990) 1-12.

\bibitem{Crespi}
Crespi, G.P.; Ginchev, I, Racca, M: A note on the minty type vector variational inequalities,  RAIRO-Oper. Res., \textbf{39} (2005) 253-273.

\bibitem{Docarmo}
Docarmo, M.P.: {\sl Riemannian geometry,} Birkhauser Boston, Boston, MA (1992).

\bibitem{Dong}
Dong, Q.L.: Linearized Douglas?Rachford method for variational inequalities with Lipschitz mappings. Comp. Appl. Math. \textbf{42}, 319 (2023)

\bibitem{Ferreira}
Ferreira, O.P.: Dini derivative and a characterization for Lipschitz and convex function on Riemannian manifolds, Nonlinear Anal. \textbf{68} (2008) 1517-1528.

\bibitem{Giannessi}
Giannessi, F.: Theorem of the alternative, quadratic programming and complementary problems. In: Cottle, R.W., Giannessi, F.; Lions, J.L. (eds.): Variational inequalities and complementary problems; Wiley, New York (1980) 151-186.

\bibitem{Hartman}
Hartman,P.; Stampacchia, G.: On some nonlinear elliptic differential functional equations. Acta math. \textbf{115} (1996) 153-188.

\bibitem{Komlosi}
Komlosi, S.: Generalized monotonicity and generalized convexity, J. Optim. Theory Appl. \textbf{84} (1995) 361-376.

\bibitem{Lee}
Lee, G.M.; Lee, K.B.: Vector variational inequalities for non-differential convex optimization problems. J. Glob. Optim. \textbf{32} (2005) 597-612.

\bibitem{Li}
Li, C.; Yao, J.C.: Variational inequalities foe set-valued vector fields on Riemannian manifolds: convexity of the solution set and the proximal point algorithm. SIAM J. Control Optim. \textbf{50} (2012) 2486-2514.

\bibitem{Mishra2}
Mishra, S.K.; Upadhyay, B.B.: Some relations between vector variational inequality problems and nonsmooth vector optimization problems using quasi efficiency. Positivity journal \textbf{17} (2013) 1071-1083.

\bibitem{Nemeth}
Nemeth, S.Z.: Variational inequalities on Hadamard manifolds. Nonlinear Anal. TMA \textbf{52} (2003) 1491-1498.

\bibitem{Sakai}
Sakai, T.: {\sl Riemannian geometry} Trans. math. monogr. 149, American mathematical society, Providence, RI (1996).

\bibitem{Rapscak}
Rapscak, T.: {\sl Smooth nonlinear optimization in $R^n$.} Kluwer, Dordrecht, The Netherlands (1997).

\bibitem{Singh}
Singh, N.; Iqbal, A.; Ali, S.: Nonsmooth vector variational inequalities on Hadamard manifold and their existence, The journal of Analysis, (2023).

\bibitem{Udriste}
Udriste, C.: {\sl Convex functions and optimization methods on Riemannian manifolds.} Math. Appl. 297, Kluwer, Dordrecht, The Netherlands , (1994).
\bibitem{Yang}
Yang, X.Q.: Vector variational inequality and vector pseudo linear optimization. J. Optim. Theory Appl. \textbf{95} (1997) 729-734.
\bibitem{Yang2}
Yang, X.Q.; Goh, C.J.: On vector variational inequality: application to vector equilibria. J. Optim. Theory Appl. \textbf{95} (1997) 413-443. 
\bibitem{Zhou}
Zhou, L.W.; Huang, N.J.: Existence of solutions for vector optimization on Hadamard manifolds. J. Optim. Theory Appl. \textbf{157} (2013) 44-53.


\end{thebibliography}
\end{document}